\documentclass[11pt,leqno]{article}

\usepackage{amsthm,amsfonts,amssymb,amsmath,color}

\numberwithin{equation}{section}

\renewcommand\d{\partial}
\renewcommand\a{\alpha}

\renewcommand\o{\omega}

\newcommand\R{\mathbb R}\newcommand\N{\mathbb N}\newcommand\Z{\mathbb Z}

\def\O{\Omega}

\def\k{\kappa}

\def\l{\lambda}
\def\vp{\varphi}
\def\epsilon{\varepsilon}
\def\e{\varepsilon}

\newcommand\br{\begin{rem}}
\newcommand\er{\end{rem}}
\newcommand\bp{\begin{pmatrix}}
\newcommand\ep{\end{pmatrix}}
\newcommand\be{\begin{equation}}
\newcommand\ee{\end{equation}}
\newcommand\ba{\begin{equation}\begin{aligned}}
\newcommand\ea{\end{aligned}\end{equation}}

\newcommand\nn{\nonumber}

\newcommand{\supp}{{\rm supp }}

\newcommand{\wto}{{\widetilde\Omega}}

\newcommand{\dive}{{\rm div\,}}

\newtheorem{defi}{Definition}[section]
\newtheorem{theorem}[defi]{Theorem}
\newtheorem{proposition}[defi]{Proposition}
\newtheorem{lemma}[defi]{Lemma}

\newtheorem{remark}[defi]{Remark}

\numberwithin{equation}{section}

\begin{document}

\title{On uniform estimates for Laplace equation in balls with small holes}

\author{Yong Lu \footnote{Mathematical Institute, Faculty of Mathematics and Physics, Charles University, Sokolovsk\'a 83, 186 75 Praha, Czech Republic, {\tt luyong@karlin.mff.cuni.cz}}
\thanks{The author thanks E. Feireisl, C. Prange, S. Schwarzacher and J. \v{Z}abensk\'y for interesting discussions. The author acknowledges the support of the project LL1202 in the program ERC-CZ funded by the Ministry of Education, Youth and Sports of the Czech Republic. }}

\date{}

\maketitle

\begin{abstract}
In this paper, we consider the Dirichlet problem of the three-dimensional Laplace equation in the unit ball with a shrinking hole. The problem typically arises from homogenization problems in domains perforated with tiny holes. We give an almost complete description concerning the uniform $W^{1,p}$ estimates: for any $3/2<p<3$ there hold the uniform $W^{1,p}$ estimates; for any $1<p<3/2$ or $3<p<\infty $, there are counterexamples indicating that the uniform $W^{1,p}$ estimates do not hold. The results can be generalized to higher dimensions.

\end{abstract}

 \tableofcontents

\renewcommand{\refname}{References}


\section{Introduction}

We consider the following Dirichlet problem of the Laplace equation with a source term of divergence form:
\ba\label{1}
-\Delta u&=\dive f,\quad &&\mbox{in}~\O_\e:=B_1\setminus \e  T,\\
u&=0,\quad &&\mbox{on} ~\d\O_\e=\d B_1\cup \e \d T.
\ea
Here $u:\O_\e \to \R$ is the \emph{unknown}, $f:\O_\e\to \R^3$ is the \emph{source function}, $\e\in (0,1)$ is a small parameter, $B_1:=B(0,1)$ is the unit ball in $\R^3$, $T$ is a closed Lipchitz subdomain of $B_1$ and is independent of $\e$.

\medskip

Our first theorem states:
\begin{theorem}\label{thm} For any $3/2<p<3$ and any $f\in L^p(\Omega_\e;\R^3)$, the unique solution $u\in W_0^{1,p}(\O_\e)$ to \eqref{1} satisfies the estimate:
\be\label{est}
\|\nabla u\|_{L^p(\Omega_\e)}\leq C \ \|f\|_{L^p(\O_\e)}
\ee
for some $C=C(p)$ independent of $\e$.

\end{theorem}

We give a remark concerning the well-posedness of \eqref{1} for any fixed $\e$. We refer to Theorem 0.5 and Theorem 1.1 in \cite{JK} for more details and the proof.
\begin{remark}\label{rem:exi}

For any Lipchitz domain $\O\subset \R^d,~d\geq 3$, there exits $p_1>3$ such that for any $p_1'<p<p_1$ and any $h\in W^{-1,p}(\O)$, the Dirichlet problem of the Laplace equation
\ba
-\Delta w&= h,\quad &&\mbox{in}~\O,\\
w&=0,\quad &&\mbox{on} ~\d\O\nn
\ea
 is well-posed in $W_0^{1,p}(\O)$ and the solution $w$ satisfies
\be\label{est-classical}
\|w\|_{W^{1,p}_0(\Omega)}\leq C(p,d,\O) \ \|h\|_{W^{-1,p}(\Omega)}.
\ee
Moreover, if the domain $\O$ is $C^1$, one can take $p_1=\infty$.

\medskip

Here and in the sequel, we use the notation $p'$ to denote the Lebesgue conjugate component of $p\in [1,\infty]$ such that $1/p'+1/p=1$;  we use $W^{-1,p}(\O)$ to denote the dual space of $W_0^{1,p'}(\O)$ for any $1<p<\infty$ and any domain $\O\subset \R^d$. The definition of the norm is classical:
$$
\|u\|_{W^{-1,p}(\O)}:=\sup_{\phi \in C_c^\infty(\O),\,\|\phi\|_{W^{1,p'}=1}}|\langle u,\phi \rangle|.
$$

\end{remark}

Our concern is the estimate for the constant $C(p,d,\O)$ appearing in \eqref{est-classical}.  If $p=2$, one has $C(2,d,\O)$=1. For $p\neq 2,$ the constant $C(p,d,\O)$ depends on the Lipchitz character of the domain $\O$. For our case, the Lipshitz norm of $\Omega_\e$ is of order $1/\epsilon$ which is unbounded when $\e\to 0$. Thus one cannot apply the classical results as in Remark \ref{rem:exi} to obtain the uniform estimate \eqref{est}.

\medskip

Our second theorem shows that the choice range of $p$ in Theorem \ref{thm} is critical:
\begin{theorem}\label{thm1}

\begin{itemize}
\item There exits $f\in C^\infty(\overline B_1;\R^3)$ such that for any $3<p<\infty$, if there exist solutions $u_\e\in W^{1,p}_0(\O_\e)$ to \eqref{1} for all $0<\e\ll 1$, then
\be\label{est1}
\liminf_{\e \to 0}\|\nabla u_\e\|_{L^p(\Omega_\e)} =\infty.
\ee

\item Suppose furthermore that $T$ has $C^1$ boundary. Then for any $1<p<3/2$ and any $0<\e<1$, there exists $f_\e\in L^p(\Omega_\e;\R^3)$ satisfying $\|f_\e\|_{L^p(\Omega_\e)}=1$ such that the unique solution $u_\e\in W^{1,p}_0(\Omega_\e)$ to \eqref{1} with source function $f_\e$ satisfies
\be\label{est11}
\liminf_{\e \to 0}\|\nabla u_\e\|_{L^p(\Omega_\e)} =\infty.
\ee
\end{itemize}
\end{theorem}

In fact, we will prove the following more general result and the first part of Theorem \ref{thm1} is a corollary of it.
\begin{theorem}\label{thm11}
 Let $f\in L^p(B_1;\R^3)$ for some $p>3$ such that $\dive f \in L^{q}(B_1)$ for some $q>3/2$. Suppose that $f$ is independent of $\e$ and satisfies
\be\label{f-ex}
\int_{B_1} \left(\frac{1}{|y|}-1\right) \dive f(y)dy\neq 0.
\ee
Then if there exist solutions $u_\e\in W^{1,p}_0(\O_\e)$ to \eqref{1} for all $0<\e\ll 1$, there hods
\be\label{est12}
\liminf_{\e\to 0}\|\nabla u_\e\|_{L^p(\Omega_\e)} =\infty.
\ee
\end{theorem}

Concerning the well-posedness of \eqref{1} in $W_0^{1,p}(\O_\e)$ with $3<p<\infty$ or $1<p<3/2$, we refer to Remark \ref{rem:exi}.

\subsection{Motivation}

The Dirichlet problem in the unit ball with a small hole arises typically in the homogenization problems in domains  perforated with very tiny holes (obstacles) for which the diameters are much smaller than their mutual distances.

The homogenization of elliptic systems and the homogenization problems in the framework of fluid mechanics have gained a lot interest: J\"{a}ger and Mikeli\'{c} \cite{JM} for the Laplace equation, Allaire \cite{ALL-NS1} and \cite{ALL-NS2} for the Stokes and stationary incompressible Navier-Stokes equations, Mikeli\'{c} \cite{Mik} for  the incompressible evolutionary Navier-Stokes equations, Masmoudi \cite{Mas-Hom} for the compressible Navier-Stokes system, Feireisl, Novotn\'y and Takahashi \cite{FNT-Hom} for the complete Navier-Stokes-Fourier system and recently Feireisl and Lu \cite{FL1} for the stationary compressible Navier-Stokes system.

Allaire in \cite{ALL-NS1,ALL-NS2} showed that the homogenization process crucially depends on the size of the holes. Specifically, for three-dimensional Stokes and stationary incompressible Navier-Stokes equations in a domain perforated with holes of diameter $O(\e^{\alpha})$, where $\e$ is the size of their mutual distances, Allaire showed that when $\alpha<3$, the behavior of the limit fluid is governed by the classical Darcy's law; when $\alpha=3$, in the limit it yields Brinkman law; when $\alpha>3$, the equations do not change in the homogenization process and the limit homogenized system is the same system of Stokes or Navier-Stokes equations.

A key point of Allaire's argument is the construction of the \emph{restriction operator} $R_\e$ which is a linear mapping from $W_0^{1,2}(\Omega)$ where $\O$ is the domain without holes to $W_0^{1,2}(\Omega_\e)$ where $\Omega_\e$ is the domain with holes. In the construction of $R_\e$ (see Section 2.2 in \cite{ALL-NS1}), there arises the Dirichlet problem of the Stokes equation in the neighborhood of any single hole. Since the holes are of diameter $O(\e^{\alpha})$ and their mutual distances are of size $\e$, then after rescalling by $\e$, there comes the Dirichlet problem of the Stokes equation in a domain of the form $B_1\setminus \e^{\a-1}T.$ The operator norm of $R_\e$ depends on the $W^{1,2}$ estimate of the Dirichlet Stokes problem in $B_1\setminus \e^{\a-1}T.$

 In the framework of $L^2$, the uniform $W^{1,2}$ estimate for elliptic equations is rather direct with the estimate constant to be $1$. However, the $L^p$ framework and $W^{1,p}$ estimate for general $p$ are needed in the homogenization of more complicated systems, such as the
evolutionary Navier-Stokes equations in \cite{Mik}, the compressible Navier-Stokes system in \cite{Mas-Hom}, and the complete Navier-Stokes-Fourier system in \cite{FNT-Hom}. In the framework of $L^p$, the estimate constant usually depends on the domain, for example, the Lipchitz character of the domain.

However, it is considered only the case $\alpha=1$ in \cite{Mik}, \cite{Mas-Hom} and \cite{FNT-Hom}, meaning that the size of holes is proportional to their mutual distances. In this case, the domain $B_1\setminus \e^{\a-1}T$=$B_1\setminus T$ is independent of $\e$. Consequently, the $W^{1,p}$ estimates can be obtained by applying the classical results, see for instance  \cite{JK} for the Laplace equation, \cite{CP} for elliptic equations in divergence form with variable coefficients and \cite{BS} for the Stokes equation.

To extend the study of homogenization problems for evolutionary Navier-Stokes equations with different size of holes, it is motivated to study the Laplace and Stokes equations in domains of the type $B_1\setminus \e T$.

\subsection{Generalization to higher dimensions}

Our results can be generalized to higher dimensions. In particular, if $T$ is a closed $C^1$ subdomain of $B_1:=\{x\in\R^d:|x|<1\},\,d\geq 4,$ and $T$ is independent of $\e$, the Dirichlet problem of the Laplace equation
\ba\label{1-d}
-\Delta u&=\dive f,\quad &&\mbox{in}~\O_\e:=B_1\setminus \e  T,\\
u&=0,\quad &&\mbox{on} ~\d\O_\e=\d B_1\cup \e \d T
\ea
admits a unique solution $u\in W^{1,p}_0(\O_\e)$ provided $f\in L^p(\O_\e)$. This is true for any $0<\e<1$ and any $1<p<\infty$. Moreover, we have the following results concerning the uniform $W^{1,p}$ estimates:
\begin{theorem}\label{thm-d} For any $d'<p<d$ and any $f\in L^p(\Omega_\e;\R^d)$,  the unique solution $u\in W_{0}^{1,p}(\O_\e)$ to \eqref{1-d} satisfies the estimate:
\be\label{est-d}
\|\nabla u\|_{L^p(\Omega_\e)}\leq C \ \|f\|_{L^p(\O_\e)}\nn
\ee
for some $C=C(p,d)$ independent of $\e$.

\end{theorem}

\begin{theorem}\label{thm1-d}

\begin{itemize}
\item There exits $f\in C^\infty(\overline B_1;\R^d)$ such that the unique solution $u_\e$ to \eqref{1-d} satisfies
\be\label{est1-d}
\liminf_{\e \to 0}\|\nabla u_\e\|_{L^p(\Omega_\e)} =\infty,\quad \mbox{for any $d<p<\infty$}.\nn
\ee

\item For any $1<p<d'$ and any $0<\e<1$, there exists $f_\e\in L^p(\Omega_\e;\R^d)$ satisfying $\|f_\e\|_{L^p(\Omega_\e)}=1$ such that the unique solution $u_\e\in W^{1,p}_0(\Omega_\e)$ to \eqref{1-d} with source function $f_\e$ satisfies
\be\label{est11-d}
\liminf_{\e \to 0}\|\nabla u_\e\|_{L^p(\Omega_\e)} =\infty.\nn
\ee
\end{itemize}
\end{theorem}

\begin{theorem}\label{thm11-d}
Let $f\in L^p(B_1;\R^d)$  for some $d<p<\infty$ such that $\dive f \in L^{q}(B_1)$ for some $q>d/2$. Suppose that $f$ is independent of $\e$ and satisfies
\be\label{f-ex-d}
\int_{B_1} \left(\frac{1}{|y|^{d-2}}-1\right) \dive f(y)dy\neq 0.\nn
\ee
Then the unique solution $u_\e\in W^{1,p}_0(\O_\e)$  to \eqref{1-d} satisfies
\be\label{est12-d}
\liminf_{\e\to 0}\|\nabla u_\e\|_{L^p(\Omega_\e)} =\infty.\nn
\ee
\end{theorem}

\medskip

We give a remark for the case where the boundary $\d T$ is only Lipchitz.
\begin{remark}
If $\d T$ is only Lipchitz, the conclusion in Theorem \ref{thm-d} holds for $p_1'<p<p_1$ for some $p_1>3$. Such a choice range of $p$ is due to the restriction on the well-posedness results to Dirichlet problem \eqref{1-d} in Sobolev spaces $W^{1,p}_0(\O_\e)$ when the domain $\O_\e$ is only Lipchitz (see Remark \ref{rem:exi}).  Accordingly there are modified versions for Theorem \ref{thm1-d} and Theorem \ref{thm11-d}.

\end{remark}

The proof for higher-dimensional case is the same as for the three-dimensional case, so we do not repeat.

\bigskip

The paper is organized as follows: Section 2 and Section 3 are devoted to the proof of Theorem \ref{thm}; Section 4 and Section 5  are devoted to the proof of Theorem \ref{thm11} and Theorem \ref{thm1}, respectively. We give some final remarks in Section 6.

\medskip

In the sequel, $C$ denotes always a constant independent of $\e$ unless there is a specification.

\section{Reformulation}

To study the uniform estimates of the Dirichlet problem \eqref{1} and prove Theorem \ref{thm}, we turn to study the following Dirichlet problem in the rescaled domain:
 \ba\label{2}
-\Delta v&=\dive g,\quad &&\mbox{in}~\widetilde\O_\e,\\
v&=0,\quad &&\mbox{on} ~\d\wto_\e,
\ea
where
\be\label{net-domain}
\widetilde \O_\e:=\O_\e/\e=B_{1/\e}\setminus T, \quad B_{1/\e}:=B(0,1/\e):=\{x\in \R^3:|x|<1/\e\}.
\ee

We have:
\begin{theorem}\label{thm2}
Let $3/2<p<3$ and $g\in L^p(\widetilde\Omega_\e;\R^3)$. Then the unique solution $v\in W_0^{1,p}(\wto_\e)$ to Dirichlet problem \eqref{2} satisfies the estimate:
\be\label{est2}
\|\nabla v\|_{L^p(\wto_\e)}\leq C \ \|g\|_{L^p(\wto_\e)}
\ee
for some $C=C(p)$ independent of $\e$.
\end{theorem}

We claim:
 \begin{proposition}\label{prop:equiv}
 Theorem \ref{thm} and Theorem \ref{thm2} are equivalent.
 \end{proposition}
 \begin{proof}[Proof of Proposition \ref{prop:equiv}]
   We suppose that Theorem \ref{thm2} holds and we want to prove Theorem \ref{thm}. Let $u\in W^{1,p}_0(\O_\e)$ be the unique solution to \eqref{1} with source function $f$ under the assumptions in Theorem \ref{thm}. We need to show the uniform estimate \eqref{est}. We rescale in the variable $x$ to define
\be\label{va-change1}
\tilde u(\cdot):= u(\e \cdot),\quad  \tilde f(\cdot):=f(\e \cdot).
\ee
Then $\tilde u$ and $\tilde f$ are functions defined in $\wto_\e$ and there holds
 \ba\label{eq-change1}
-\Delta \tilde u&=\e\,\dive \tilde f,\quad &&\mbox{in}~\widetilde\O_\e,\\
\tilde u&=0,\quad &&\mbox{on} ~\d\wto_\e.
\ea
We apply Theorem \ref{thm2} to Dirichlet problem \eqref{eq-change1} to obtain
\be\label{est-change1}
\|\nabla \tilde u\|_{L^p(\wto_\e)}\leq C \e\ \|\tilde f\|_{L^p(\wto_\e)}.
\ee

 Then back to the original variable through \eqref{va-change1}, it gives
\be\label{est-change2}
\|\nabla  u\|_{L^p(\O_\e)}\leq C \ \|f\|_{L^p(\O_\e)}.
\ee
The constant $C=C(p)$ in \eqref{est-change1} and \eqref{est-change2} is the same as in Theorem \ref{thm2}, which is independent of $\e$. Thus we proved Theorem \ref{thm}.

\medskip

Proving Theorem \ref{thm2} by assuming Theorem \ref{thm} can be done similarly. We complete the proof of Proposition \ref{prop:equiv}.

\end{proof}

Hence, to prove Theorem \ref{thm}, it is sufficient to prove Theorem \ref{thm2}. This is done in the next section.

\section{Proof of Theorem \ref{thm2}}

This section is devoted to the proof of Theorem \ref{thm2}. At the same time we will have proven Theorem \ref{thm} due to Proposition \ref{prop:equiv}. Inspired by the idea in \cite{KS}, we decompose the Dirichlet problem \eqref{2} into two parts by using some cut-off function. The first part is defined in a bounded Lipchitz domain, so we can employ classical results to obtain uniform estimates. The other part is defined in the enlarging ball $B_{1/\e}$, and we employ the Green's function of Laplace equation to get uniform estimates. In particular, in Section \ref{sec:Be} we show some general results concerning the Dirichlet problem in the ball $B_{1/\e}$ in $\R^d$. These results may be of independent interest.

\medskip

 We assume $0<\e\leq 1/4$ in the sequel for the convenience of defining cut-off functions; otherwise for $ 1/4< \e <1$ the result in Theorem \ref{thm2} is rather classical (see for instance Theorem 0.5 in \cite{JK}).

\subsection{Decomposition}
We introduce the cut-off function:
\be\label{cut-off}
\vp\in C_c^\infty (B_2),\ B_2:=B(0,2),\quad \vp\equiv 1 \ \mbox{in}\  B_1\supset T,\quad 0\leq \vp\leq 1.
\ee

Let $v\in W^{1,p}_0(\wto_\e)$ be the unique solution to \eqref{2} under the assumptions in Theorem \ref{thm2}. We consider the decomposition:
\be\label{v12}
 v=v_1+v_2,\quad v_1:=\vp v, \quad v_2 :=(1-\vp)v.
\ee
Then $v_1$ and $v_2$ solve respectively
 \ba\label{eq-v1}
-\Delta v_1&=\dive (g\vp)-(v\Delta \vp+2\nabla v \nabla \vp+g \nabla \vp),\quad &&\mbox{in}~B_2\setminus T,\\
v_1&=0,\quad &&\mbox{on} ~\d B_2\cup \d T
\ea
and
 \ba\label{eq-v2}
-\Delta v_2&=\dive (g(1-\vp))+(v\Delta \vp+2\nabla v \nabla \vp+g \nabla \vp),\quad &&\mbox{in}~B_{1/\e},\\
v_2&=0,\quad &&\mbox{on} ~\d B_{1/\e}.
\ea
Here, we treat $v_1$ as the solution of the Dirichlet problem in the bounded domain $B_2\setminus T$ and $v_2$ as the solution of the Dirichlet problem in the enlarging ball $B_{1/\e}$.



\subsection{Dirichlet problem in bounded domain}\label{sec:v1}

In this section, we consider the the Dirichlet problem \eqref{eq-v1}. Since the domain $B_2\setminus T$ is bounded and Lipchitz, we can employ Theorem 0.5 in \cite{JK} (see also Remark \ref{rem:exi}) to obtain
\be\label{est-v1}
\|v_1\|_{W^{1,p}_0(B_2\setminus T)}\leq C\ \|\dive (g\vp)-(v\Delta \vp+2\nabla v \nabla \vp+g \nabla \vp)\|_{W^{-1,p}(B_2\setminus T)}.
\ee

We estimate the right-hand side of \eqref{est-v1} term by term. Let $\psi \in C_c^\infty(B_2\setminus T)$ be an arbitrary test function, then
\ba\label{est-v11}
&|\langle\dive (g\vp),\psi\rangle|=|\langle(g\vp),\nabla \psi\rangle| \leq \|g\vp\|_{L^p}\|\nabla\psi\|_{L^{p'}}\leq \|g\|_{L^p}\|\nabla\psi\|_{L^{p'}},\\
&|\langle v\Delta \vp,\psi\rangle|\leq \|v\Delta\vp\|_{L^p}\|\psi\|_{L^{p'}}\leq C \|v\|_{L^p}\|\psi\|_{L^{p'}},\\
&|\langle g \nabla \vp,\psi \rangle|\leq \|g\nabla\vp \|_{L^p} \|\psi\|_{L^{p'}}\leq C\|g\|_{L^p} \|\psi\|_{L^{p'}},\\
&|\langle \nabla v\nabla \vp,\psi\rangle|=|\langle\nabla v,\nabla \vp\psi\rangle|=|\langle v, \dive(\nabla \vp\psi)\rangle|=|\langle v, \Delta\vp \psi+\nabla\vp\nabla\psi \rangle |\\
&\quad \leq \|v\|_{L^p}\|\Delta\vp \psi+\nabla\vp\nabla\psi\|_{L^{p'}}\leq C\|v\|_{L^p} \|\psi\|_{W^{1,p'}}.
\ea
In \eqref{est-v11}, the Lebesgue norms are taken in the domain $B_2\setminus T$. The estimates in \eqref{est-v1} and \eqref{est-v11} imply
\be\label{est-v12}
\|v_1\|_{W^{1,p}_0(B_2\setminus T)}\leq C \left(\|v\|_{L^p(B_2\setminus T)}+\|g\|_{L^p(B_2\setminus T)}\right).
\ee

\subsection{Dirichlet problem in enlarging balls}\label{sec:Be}

In this section, we consider the Dirichlet problem of the Laplace equation in $B_{1/\e}\subset \R^d,~d\geq 3$ and we will show some general results which may be of independent interest. The problem reads:
 \ba\label{eq-ow}
-\Delta \o&=\pi,\quad &&\mbox{in}~B_{1/\e}:=\{x\in\R^d:|x|<1/\e\},\\
\o&=0,\quad &&\mbox{on} ~\d B_{1/\e}.
\ea

Our first result concerns the case where the source term $\pi$ is of divergence form:
\begin{lemma}\label{lem:lap-Be0} If $\pi=\dive \eta$ for some $\eta \in L^q(B_{1/\e};\R^d)$ with $q\in (1,\infty)$, then the unique solution $\o$ to \eqref{eq-ow} satisfies
\be\label{est-ow0}
\|\nabla \o\|_{L^q(B_{1/\e})}\leq C\, \|\eta\|_{L^q(B_{1/\e})}
\ee
for some constant $C=C(q,d)$ independent of $\e$.
\end{lemma}

\begin{proof}[Proof of Lemma \ref{lem:lap-Be0}] The proof of Lemma \ref{lem:lap-Be0} is similar as the proof of Proposition \ref{prop:equiv}. We introduce the change of variables up to a rescalling by $\e$:
$$
\tilde \o(\cdot):=\o(\frac{\cdot}{\e}),\quad \tilde \eta (\cdot):=\eta(\frac{\cdot}{\e}).
$$
Then $(\tilde \o,\tilde \eta)$ solves
 \ba\label{eq-w11}
-\Delta \tilde \o&=\e^{-1}\dive \tilde \eta,\quad &&\mbox{in}~B_{1},\\
\tilde \o&=0,\quad &&\mbox{on} ~\d B_{1}.
\ea
By Theorem 0.5 in \cite{JK}, we deduce
$$
\|\nabla \tilde \o\|_{L^q(B_{1})}\leq C\,\e^{-1} \|\tilde \eta \|_{L^q(B_{1})}
$$
for some $C=C(p,d)$ independent of $\e$. Back to the original variables, we obtain
\be\label{est-w10}
\|\nabla \o\|_{L^q(B_{1/\e})}\leq C \,\|\eta \|_{L^q(B_{1/\e})}.\nn
\ee
 The proof is completed.

\end{proof}

Our second result concerns the case where $\pi$ is compactly supported:
\begin{lemma}\label{lem:lap-Be}
Suppose $d'<q<d$ and $\pi\in L^1\cap W^{-1,q}(\R^d)$ having a compact support that is independent of $\e$. Then the unique solution $\o$ to \eqref{eq-ow}
satisfies
\be\label{est-ow}
\|\nabla \o\|_{L^q(B_{1/\e})}\leq C \left(\|\pi\|_{W^{-1,q}(\R^d)}+\e^{d\left(1-\frac{1}{q}\right)}\|\pi\|_{L^1}\right)
\ee
for some constant $C=C(q,d)$ independent of $\e$.

\end{lemma}

\begin{proof}[Proof of Lemma \ref{lem:lap-Be}]Without loss of generality, we assume $0<\e<1/4$ and
$$
{\rm supp}\, \pi \subset B_2:=\{x\in \R^d:|x|<2\}.
$$

We recall the Green's function of the Laplace equation in the ball $B_{1/\e}$:
\be\label{green}
G_\e(x,y)=\Phi(x-y) -\Phi \left(\e|x|\left(\frac{x}{\e^2|x|^2}-y\right)\right),
\ee
where $\Phi(x)=\alpha/|x|^{d-2}$ is the fundamental solution of the Laplace operator in $\R^d,~d\geq 3$. One can derive \eqref{green} by using the Green's function of the Laplace equation in the unit ball
\be\label{green-ball}
G(x,y)=\Phi(x-y) -\Phi \left(|x|\left(\frac{x}{|x|^2}-y\right)\right),\nn
\ee
and the fact that
$$
G_\e(x,y)=\Phi(x-y)-\phi^y(x),
$$
where $\phi^y(x)$ is the solution to
$$\Delta_x \phi^y(x)=0 \quad \mbox{in}\ B_{1/\e}, \quad \phi^y(x)=\Phi(x-y)\quad \mbox{on}\ \d B_{1/\e}.$$

\smallskip

By employing the Green's function, we can write the solution $\o$ to \eqref{eq-ow} as
\be\label{w2-green}
\o(x)=\int_{B_{1/\e}} G_\e(x,y)  \pi(y) dy=m_1(x)+m_2(x),\nn
\ee
where
\ba\label{m12}
&m_1(x):=\int_{B_{2}} \Phi(x-y) \pi(y) dy=(\Phi\ast \pi)(x),\\
&m_2(x):=-\int_{B_{2}} \Phi \left(\e|x|\left(\frac{x}{\e^2|x|^2}-y\right)\right) \pi (y) dy.\nn
\ea

Let $\psi\in C_c^\infty(B_{1/\e};\R^d)$ be an arbitrary test function and $\chi\in C_c^\infty(B_2)$ be a cut-off function such that
$$
\chi\equiv 1 \quad \mbox{on}\ {\rm \supp}\,\pi , \quad 0\leq \chi\leq 1.
$$
Then we have
\ba\label{est-m1}
&|\langle\nabla m_1,\psi \rangle|=|\langle\Phi\ast \pi ,\dive \psi \rangle|= |\langle \pi ,\Phi\ast \dive \psi \rangle|=|\langle \pi ,\dive (\Phi\ast \psi) \rangle|\\
&=|\langle \pi ,\chi\dive (\Phi\ast \psi) \rangle|\leq \|\pi \|_{W^{-1,q}(\R^d)}\,\|\chi\dive (\Phi\ast \psi)\|_{W^{1,q'}(B_2)}\\
&\leq C\,\|\pi \|_{W^{-1,q}(\R^d)}\,\|\dive (\Phi\ast \psi)\|_{W^{1,q'}(B_2)}.
\ea
We consider
\ba\label{est-Phi1}
&\|\dive (\Phi\ast \psi)\|_{L^{q'}(B_2)}\leq C\,\sum_{j=1}^d\|(\d_{j}\Phi\ast \psi)\|_{L^{q'}(B_2)}\\
&\quad \leq C\,\sum_{j=1}^d\left(\big\|(1_{B_6}\d_{j}\Phi)\ast \psi\big\|_{L^{q'}(B_2)}+\big\|(1_{B_6^c}\d_{j}\Phi)\ast \psi\big\|_{L^{q'}(B_2)}\right),
\ea
where $1_{B_6}$ and $1_{B_6^c}$ are character functions defined as
$$
1_{B_6}(x)=1\quad \mbox{for}\quad x\in B_6;\quad 1_{B_6}(x)=0\quad \mbox{for}\quad x\in B_6^c;\quad 1_{B_6}+1_{B_6^c}\equiv 1,
$$
where we used the notations
$$B_6:=\{x\in \R^d,\ |x|<6\},\quad B_6^c:=\{x\in \R^d,\ |x|\geq 6\}.$$

\medskip

Young's inequality implies
\ba\label{est-Phi2}
\big\|(1_{B_6}\d_{j}\Phi)\ast \psi\big\|_{L^{q'}(B_2)}\leq C\,\big\|(1_{B_6}\d_{j}\Phi)\big\|_{L^{1}(\R^d)}\big\| \psi\big\|_{L^{q'}(\R^d)}\leq C\,\big\| \psi\big\|_{L^{q'}(\R^d)},
\ea
where we used the fact that
\be\label{est-Phi3}
\d_{j}\Phi=-\a(d-2)\frac{x_j}{|x|^d}\in L^1_{loc}(\R^d).\nn
\ee

\medskip

We then calculate
\be\label{est-Phi4}
\big\|(1_{B_6^c}\d_{j}\Phi)\ast \psi\big\|_{L^{q'}(B_2)}=\a(d-2)\left(\int_{|x|\leq 2}\left|\int_{|x-y|\geq 6} \frac{x_j-y_j}{|x-y|^d}\, \psi(y)\,dy \right|^{q'}dx\right)^{\frac {1}{q'}}.
\ee
For any $(x,y)$ such that $|x|\leq 2$ and $|x-y|\geq 6$, there holds
\be\label{est-Phi5}
|y|\geq |x-y|-|x|\geq 4\geq 2|x|,\quad |x-y|\geq |y|-|x|\geq |y|/2.
\ee
Then by \eqref{est-Phi4} and \eqref{est-Phi5}, we have
\ba\label{est-Phi6}
&\big\|(1_{B_6^c}\d_{j}\Phi)\ast \psi\big\|_{L^{q'}(B_2)}\leq C\left(\int_{|x|\leq 2}\left|\int_{|y|\geq 4} \frac{1}{|y|^{d-1}} \,|\psi(y)| \,dy \right|^{q'}dx\right)^{\frac {1}{q'}}\\
&\quad \leq C\left(\int_{|x|\leq 2}\left|\int_{|y|\geq 4} \frac{1}{|y|^{(d-1)q}} \,dy \right|^{\frac {q'}{q}} \left| \int_{|y|\geq 4} |\psi(y)|^{q'}\,dy \right| dx\right)^{\frac {1}{q'}}\\
&\quad \leq C \,\|\psi\|_{L^{q'}(\R^d)},
\ea
where we used the fact that
\be\label{est-Phi7}
q>d'=\frac{d}{d-1},\quad (d-1)q>d,\quad \left|\int_{|y|\geq 4} \frac{1}{|y|^{(d-1)q}} \,dy \right|<\infty.
\ee

Thus, the estimates \eqref{est-Phi1}, \eqref{est-Phi2} and \eqref{est-Phi6} imply
\be\label{est-Phi-f}
\|\dive (\Phi\ast \psi)\|_{L^{q'}(B_2)}\leq C\,\|\psi\|_{L^{q'}(\R^d)}=C\,\|\psi\|_{L^{q'}(B_{1/\e})}.
\ee

\medskip

By the classical Calder\'on-Zygmund theorem, direct calculation gives
\be\label{est-m10}
\|\nabla \dive (\Phi\ast \psi)\|_{L^{q'}(\R^d)}\leq \sum_{i,j=1}^d \| (\d_i\d_j\Phi)\ast \psi\|_{L^{q'}(\R^d)}\leq C\, \|\psi\|_{L^{q'}(\R^d)}=C\, \|\psi\|_{L^{q'}(B_{1/\e})}.
\ee
In fact, the convolution operator
$$
(\d_i\d_j\Phi)\ast \psi=\mathcal{R}_i\mathcal{R}_j \psi,
$$
where $\mathcal{R}_i,~i\in \{1,2,\cdots,d\}$ are the Riesz operators which are bounded from $L^q(\R^d)$ to $L^q(\R^d)$ for any $1<q<\infty$.

By \eqref{est-m1}, \eqref{est-Phi-f} and \eqref{est-m10}, we obtain
\be\label{est-m12}
\|\nabla m_1\|_{L^{q}(B_{1/\e})}\leq C\, \|\pi \|_{W^{-1,q}(\R^d)}.
\ee

\bigskip

For $m_2$,  we recall its definition from \eqref{m12}:
\be\label{def-m2}
 m_2(x)=-\frac{\alpha }{\e^{d-2}|x|^{d-2}}\int_{B_2} \frac{\pi (y)}{\left|\frac{x}{\e^2|x|^2}-y\right|^{d-2}} \,dy.
\ee
Then
\ba\label{est-m2}
&\d_{x_j} m_2(x)=\frac{\alpha (d-2)x_j}{\e^{d-2}|x|^d}\int_{B_2} \frac{\pi (y)}{\left|\frac{x}{\e^2|x|^2}-y\right|^{d-2}} \,dy\\
&\qquad-\frac{\alpha }{\e^{d-2}|x|^{d-2}}\int_{B_2}\pi (y)\d_{x_j}\left(\left|\frac{x}{\e^2|x|^2}-y\right|^{2-d}\right)\, dy.
\ea
Direct calculation gives
\ba\label{est-m200}
&\d_{x_j}\left(\left|\frac{x}{\e^2|x|^2}-y\right|^{2-d}\right)=(2-d)\left|\frac{x}{\e^2|x|^2}-y\right|^{-d}
\left(-\frac{x_j}{\e^4|x|^4}-\frac{y_j}{\e^2|x|^2}+\frac{2x_j(x\cdot y)}{\e^2|x|^4}\right).\nn
\ea
Then we have
\ba\label{est-m201}
&\nabla m_2(x)=\frac{\alpha (d-2)x}{\e^{d-2}|x|^d}\int_{B_2} \frac{\pi (y)}{\left|\frac{x}{\e^2|x|^2}-y\right|^{d-2}} \,dy - \frac{\alpha (d-2)x}{\e^{d+2}|x|^{d+2}}\int_{B_2}\frac{\pi (y)}{\left|\frac{x}{\e^2|x|^2}-y\right|^{d}}\, dy\\
&\qquad - \frac{\alpha (d-2)}{\e^{d}|x|^{d}}\int_{B_2}\frac{\pi (y)y}{\left|\frac{x}{\e^2|x|^2}-y\right|^{d}}\, dy+\frac{2\alpha (d-2)x}{\e^{d}|x|^{d+2}}\int_{B_2}\frac{\pi (y)(x\cdot y)}{\left|\frac{x}{\e^2|x|^2}-y\right|^{d}}\, dy.
\ea

We consider
\ba\label{est-m202}
&I_1:=\frac{\alpha (d-2)x}{\e^{d-2}|x|^d}\int_{B_2} \frac{\pi (y)}{\left|\frac{x}{\e^2|x|^2}-y\right|^{d-2}} \,dy - \frac{\alpha (d-2)x}{\e^{d+2}|x|^{d+2}}\int_{B_2}\frac{\pi (y)}{\left|\frac{x}{\e^2|x|^2}-y\right|^{d}}\, dy\\
&\quad=\frac{\alpha (d-2)x}{\e^{d+2}|x|^{d+2}}\int_{B_2} \frac{\pi (y)}{\left|\frac{x}{\e^2|x|^2}-y\right|^{d}} \left(\e^4|x|^2 \left|\frac{x}{\e^2|x|^2}-y\right|^{2} -1\right)\,dy.
\ea
For any $y\in B_2$ and $x\in B_{1/\e}$ with $0<\e<1/4$ there holds,
\be\label{est-m203}
\left|\e^4|x|^2 \left|\frac{x}{\e^2|x|^2}-y\right|^{2} -1\right|=\left| \left|\frac{x}{|x|}-\e^2|x| y\right|^{2} -1\right|\leq 4\e
\ee
and
\be\label{est-m20}
\left|\frac{x}{\e^2|x|^2}-y\right|\geq \frac{1}{\e^2|x|}-2\geq \frac{1}{2\e^2|x|}.
\ee
By \eqref{est-m202}, \eqref{est-m203} and \eqref{est-m20}, we obtain
\be\label{est-m204}
|I_1| \leq \frac{C\e^{d-1}}{|x|} \|\pi \|_{L^1}.
\ee
Again by using \eqref{est-m20}, we have
\ba\label{est-m205}
\left|- \frac{\alpha (d-2)}{\e^{d}|x|^{d}}\int_{B_2}\frac{\pi (y)y}{\left|\frac{x}{\e^2|x|^2}-y\right|^{d}}\, dy
+\frac{2\alpha (d-2)x}{\e^{d}|x|^{d+2}}\int_{B_2}\frac{\pi (y)(x\cdot y)}{\left|\frac{x}{\e^2|x|^2}-y\right|^{d}}\, dy\right|\leq C\,\e^d \|\pi\|_{L^1}.
\ea

The estimates \eqref{est-m201}, \eqref{est-m204} and \eqref{est-m205} imply
\ba\label{est-m21}
|\nabla m_2(x)|\leq C\left(\e^{d-1}|x|^{-1}+\e^d\right) \|\pi \|_{L^1}.
\ea
Since $q<d$, we have
\ba\label{est-m22}
&\left\|\e^{d-1}|x|^{-1}\right\|_{L^q(B_{1/\e})}=\e^{d-1}\left(\int_{B_{1/\e}}|x|^{-q}dx\right)^{\frac{1}{q}}\leq C \e^{d\left(1-\frac{1}{q}\right)},\\
&\|\e^d\|_{L^q(B_{1/\e})}\leq C \e^{d\left(1-\frac{1}{q}\right)}.
\ea
By \eqref{est-m21} and \eqref{est-m22}, we finally derive
\ba\label{est-m23}
\left\|\nabla m_2\right\|_{L^q(B_{1/\e})}\leq  C \,\e^{d\left(1-\frac{1}{q}\right)}\|\pi \|_{L^1}.
\ea
We obtain \eqref{est-ow} by summing up the estimates for $m_1$ and $m_2$ in \eqref{est-m12} and \eqref{est-m23}. This completes the proof of Lemma \ref{lem:lap-Be}.
\end{proof}

\subsection{A further decomposition}

We will apply Lemma \ref{lem:lap-Be0} and Lemma \ref{lem:lap-Be} to study Dirichlet problem \eqref{eq-v2} in $v_2$. It is convenient to consider the following decomposition:
\be\label{w12}
v_2:=w_1+w_2,
\ee
where $w_1$ and $w_2$ solve respectively
 \ba\label{eq-w1}
-\Delta w_1&=\dive (g(1-\vp)),\quad &&\mbox{in}~B_{1/\e},\\
w_1&=0,\quad &&\mbox{on} ~\d B_{1/\e}
\ea
and
 \ba\label{eq-w2}
-\Delta w_2&=(v\Delta \vp+2\nabla v \nabla \vp+g \nabla \vp),\quad &&\mbox{in}~B_{1/\e},\\
w_2&=0,\quad &&\mbox{on} ~\d B_{1/\e}.
\ea

\medskip

Thus, Dirichlet problem \eqref{eq-w1} has a source term of divergence form so that we can apply Lemma \ref{lem:lap-Be0} and Dirichlet problem \eqref{eq-w2} has a source term being compactly supported so that we can apply Lemma \ref{lem:lap-Be}.

\medskip

By the properties of $\vp$ in \eqref{cut-off}, we have
$$
\|g(1-\vp)\|_{L^p(B_{1/\e})}\leq \|g\|_{L^p(\wto_\e)},
$$
Then applying Lemma \ref{lem:lap-Be0} to Dirichlet problem \eqref{eq-w1} gives:
\begin{proposition}\label{prop-w1} The unique solution $w_1$ to \eqref{eq-w1} satisfies
\be\label{est-w1}
\|\nabla w_1\|_{L^p(B_{1/\e})}\leq C\, \|g\|_{L^p(\wto_\e)}
\ee
for some $C=C(p)$ independent of $\e$.
\end{proposition}

\medskip

For the Dirichlet problem \eqref{eq-w2}, we have the following proposition by using Lemma \ref{lem:lap-Be}:
\begin{proposition}\label{prop-w2} Let
$\pi :=v\Delta \vp+2\nabla v \nabla \vp+g \nabla \vp$
be the right-hand side of equation $\eqref{eq-w2}_1$.
Then $\pi$ is compactly supported in $B_2\setminus T$ and the unique solution $w_2$ to \eqref{eq-w2} satisfies
\be\label{est-w2}
\|\nabla w_2\|_{L^p(B_{1/\e})}\leq C\, \e^{3-\frac{3}{p}}\|\pi \|_{L^1}+C\, \|\pi \|_{W^{-1,p}(B_2\setminus T)}
\ee
for some $C=C(p)$ independent of $\e$.
\end{proposition}

\begin{proof}[Proof of Proposition \ref{prop-w2}] By the choice of $\vp$ in \eqref{cut-off}, we have that ${\rm supp}\,\pi \subset (B_2\setminus B_1)\subset (B_2\setminus T)$. The estimate \eqref{est-w2} follows by applying Lemma \ref{lem:lap-Be} with $d=3$ and the fact
$$
\|\pi\|_{W^{-1,q}(\R^d)}\leq C\,\|\pi\|_{W^{-1,q}(B_2\setminus T)}.
$$

\end{proof}

\subsection{End of the proof}

Based on the estimate \eqref{est-v12}, Proposition \ref{prop-w1} and Proposition \ref{prop-w2}, we are ready to prove the following crucial proposition:
\begin{proposition}\label{prop:v}
Let $v$ be the unique solution to \eqref{2} under the assumptions in Theorem \ref{thm2}. Then there holds the estimate
\be\label{est-vw1}
\|\nabla v\|_{L^p(\wto_{\e})} \leq C \left(\|v\|_{L^p(B_2\setminus T)}+\|g\|_{L^p(\wto_\e)}\right).
\ee
\end{proposition}
\begin{proof}[Proof of Proposition \ref{prop:v}]
First of all, we consider the estimates of $\|\pi \|_{W^{-1,p}(B_2\setminus T)}$ and $\|\pi \|_{L^1}$ appearing in Proposition \ref{prop-w2}. Similar as the arguments in Section \ref{sec:v1},  particularly by the estimates in \eqref{est-v11}, we have
\be\label{est-g1}
\|\pi \|_{W^{-1,p}(B_2\setminus T)}\leq C \left(\|v\|_{L^p(B_2\setminus T)}+\|g\|_{L^p(B_2\setminus T)}\right).\nn
\ee
For $\|\pi \|_{L^1}$, direct calculation gives
\ba\label{est-g12}
\|\pi \|_{L^1}\leq C\left(\|v\|_{L^1(B_2\setminus T)}+\|\nabla v\|_{L^1(B_2\setminus T)}+\|g\|_{L^1(B_2\setminus T)}\right).\nn
\ea

Then, using Proposition \ref{prop-w2} implies
\be\label{est-w21}
\|\nabla w_2\|_{L^p(B_{1/\e})} \leq C \left(\|v\|_{L^p(B_2\setminus T)}+\|g\|_{L^p(B_2\setminus T)}\right)+C\e^{3-\frac{3}{p}}\|\nabla v\|_{L^1(B_2\setminus T)}.\nn
\ee
Together with \eqref{w12} and Proposition \ref{prop-w1}, we derive
\be\label{est-vw}
\|\nabla v_2\|_{L^p(B_{1/\e})} \leq C \left(\|v\|_{L^p(B_2\setminus T)}+\|g\|_{L^p(\wto_\e)}\right)+C\e^{3-\frac{3}{p}}\|\nabla v\|_{L^1(B_2\setminus T)}.
\ee
Then by \eqref{v12}, \eqref{est-v12} and \eqref{est-vw}, we obtain
\be\label{est-vw2}
\|\nabla v\|_{L^p(B_{\wto_\e})} \leq C \left(\|v\|_{L^p(B_2\setminus T)}+\|g\|_{L^p(\wto_\e)}\right)+C\e^{3-\frac{3}{p}}\|\nabla v\|_{L^p(B_2\setminus T)}.
\ee
Without loss of generality, we may assume $\e \leq \e_0$ where $C\e_0^{3-\frac{3}{p}}=1/2.$ For the case $\e_0<\e<1$, Theorem \ref{thm} and Theorem \ref{thm2} are rather classical.

Then for $\e\leq \e_0,$ the term $C\e^{3-\frac{3}{p}}\|\nabla v\|_{L^p(B(_2\setminus T)}$ appearing on the right-hand side of \eqref{est-vw2} can be absorbed by the left-hand side of \eqref{est-vw2}. We finally obtain \eqref{est-vw1} and complete the proof of Proposition \ref{prop:v}.

\end{proof}

Now we can prove Theorem \ref{thm2} by contradiction. We suppose that Theorem \ref{thm2} does not hold. Then there exist $p\in (3/2,3)$, a sequence $\{\e_k\}_{k\in\N}$ of positive numbers and a sequence $\{g_k\}_{k\in \N}$ of $L^{p}(\wto_{\e_k})$ functions satisfying
\be\label{gk}
\e_k\to 0,\ \mbox{as $k\to \infty$},\quad \|g_k\|_{L^{p}(\wto_{\e_k})}=1\ \mbox{for any $k\in \N$},\nn
\ee
such that the unique solution $v_k\in W_0^{1,p}(\wto_{\e_k})$ to the Dirichlet problem
 \ba\label{vk}
-\Delta v_k&=\dive g_k,\quad &&\mbox{in}~\widetilde\O_{\e_k},\\
v_k&=0,\quad &&\mbox{on} ~\d\wto_{\e_k}\nn
\ea
satisfies
\ba\label{gk1}
\|\nabla v_k\|_{L^{p}(\wto_{\e_k})} \to +\infty,\ \mbox{as $k\to \infty$}.\nn
\ea

Then the couple $(\tilde v_k,\tilde g_k)$ defined by
$$
\tilde v_k:=\frac{v_k}{\|\nabla v_k\|_{L^{p}(\wto_{\e_k})}},\quad \tilde g_k:=\frac{g_k}{\|\nabla v_k\|_{L^{p}(\wto_{\e_k})}}
$$
satisfies
\ba\label{t-gk}
\|\nabla \tilde v_k\|_{L^{p}(\wto_{\e_k})} =1 \ \mbox{for any  $k\in \N$},\quad \|\tilde g_k\|_{L^{p}(\wto_{\e_k})} \to 0 \ \mbox{as $k\to \infty$}
\ea
and
 \ba\label{eq-vk}
-\Delta \tilde v_k&=\dive \tilde g_k,\quad &&\mbox{in}~\widetilde\O_{\e_k},\\
\tilde v_k&=0,\quad &&\mbox{on} ~\d\wto_{\e_k}.
\ea
By Proposition \ref{prop:v}, the couple $(\tilde v_k,\tilde g_k)$ enjoys the stimate
\be\label{t-vg}
\|\nabla \tilde v_k\|_{L^p(\wto_{\e_k})} \leq C \left(\|\tilde v_k\|_{L^p(B_2\setminus T)}+\|\tilde g_k\|_{L^p(\wto_{\e_k})}\right).
\ee
By the uniform estimate in \eqref{t-gk} and Sobolev embedding, we have
\be\label{norm-pstar}
\sup_{k\in\N}\| \tilde v_k\|_{L^{p^*}(\wto_{\e_k})}\leq C, \quad \frac{1}{p^*}=\frac{1}{p}-\frac{1}{3}.
\ee

For any $k\in \N$, we define the zero extension of $\tilde v_k$:
\be\label{t-wk}
\tilde w_k=\tilde v_k \ \mbox{in}\ \wto_{\e_k},\quad \tilde w_k=0\  \mbox{in} \ \R^3\setminus  \wto_{\e_k}.
\ee
 Since $\tilde v_k \in W_0^{1,p}(\wto_{\e_k})$,  we have
 \be\label{t-wk2}
 \nabla \tilde w_k=\nabla \tilde v_k \ \mbox{in}\ \wto_{\e_k},\quad \nabla \tilde w_k=0\  \mbox{in} \ \R^3\setminus  \wto_{\e_k}.
 \ee

By the estimates in \eqref{t-gk} and \eqref{norm-pstar}, we have the uniform estimates for the extensions
$$
\|\nabla \tilde w_k\|_{L^p(\R^3\setminus T)}=1,\quad \sup_{k\in\N}\| \tilde w_k\|_{L^{p^*}(\R^3\setminus T)}\leq C.
$$
We then have the weak convergence
\be\label{st-con2}
\tilde w_k\to \tilde w_\infty \ \mbox{weakly in}\ L^{p^*}(\R^3\setminus T),\quad \nabla\tilde w_k\to \nabla \tilde w_\infty \  \mbox{weakly in}\ L^p(\R^3\setminus T).
\ee

Moreover, passing $k\to 0$ in the weak formulation of \eqref{eq-vk} implies that for any $\phi\in C_c^\infty(\R^3\setminus T)$, we have
$$
\int_{\R^3\setminus T} \nabla \tilde w_\infty \cdot \nabla \phi \, dx =0.
$$
This means the limit $\tilde w_\infty$ is a harmonic function in exterior domain $\R^3\setminus T$:
 \be\label{eq-w0}
-\Delta \tilde w_\infty=0,\quad  \mbox{in}~ \R^3\setminus T.
\ee
Then $\tilde w_\infty$ is smooth in $\R^3\setminus T$. Moreover, since $\tilde w_k$ have zero trace on $\d T$, we have
\be\label{eq-w01}
\tilde w_\infty=0 \ \mbox{on}\ \d T.
\ee
Since $\tilde w_\infty \in\ L^{p^*}(\R^3\setminus T)$, we have
\be\label{eq-w02}
\lim_{|x|\to \infty}\tilde w_\infty(x)=0.
\ee

By the maximal principle of harmonic functions, we derive
$$
\tilde w_\infty=0 \quad \mbox{in}\ \R^3\setminus T.
$$

\medskip

On the other hand,  the fact that $B_2\setminus T$ is a bounded Lipchitz subdomain of $\wto_{\e_k}$ implies
\be\label{loc-sob}
\sup_{k\in\N}\|\tilde v_k\|_{W^{1,p}(B_2\setminus T)} \leq C.
\ee
By virtute of the Rellich-Kondrachov compact embedding theorem, up to a substraction of subsequence, we have
\be\label{st-con1}
\tilde w_k = \tilde v_k\to \tilde w_\infty=0 \quad  \mbox{strongly in}\quad L^q(B_2\setminus T)\quad \mbox{for any $1\leq q<p^*$}.
\ee

Hence, passing $k\to \infty$ in \eqref{t-vg} implies the following contradiction:
$$
1\leq 0.
$$

 This implies that Theorem \ref{thm2} is true and we complete the proof.

\section{Proof of Theorem \ref{thm11}}

In this section, we prove Theorem \ref{thm11} by contradiction. Let $f\in L^p(B_1;\R^3),\, p>3$ be as in Theorem \ref{thm11} and $u_\e\in W_0^{1,p}(\O_\e)$ be a solution to \eqref{1} for any $0<\e\ll 1$. By contradiction we suppose that
\be\label{est-con}
\liminf_{\e\to 0}\|\nabla u_\e\|_{L^p(\Omega_\e)} <\infty.
\ee
Then there exists a subsequence $\{\e_k\}_{k\in \N}$ such that $\e_k\to 0$ as $\k\to \infty$ and
\be\label{est-con1}
\sup_{k\in \N}\|\nabla u_{\e_k}\|_{L^p(\Omega_{\e_k})} <\infty.
\ee

We consider the zero extension of $u_{\e_k}$:
$$
\tilde u_{\e_k} =u_{\e_k} \ \mbox{in}\  \O_{\e_k},\quad \tilde u_{\e_k}=0 \ \mbox{on}\ {\e_k} T.
$$
Then $\tilde u_{\e_k} \in W_0^{1,p}(B_1)$ and
$$
\nabla\tilde u_{\e_k} =\nabla u_{\e_k}\ \mbox{in}\  \O_{\e_k},\quad \nabla \tilde u_{\e_k}=0 \ \mbox{on}\ {\e_k} T.
$$
Therefore by \eqref{est-con1}, we have
\be\label{est-uek}
\sup_{k\in \N} \|\tilde u_{\e_k}\|_{W_0^{1,p}(B_1)}\leq C\, \sup_{k\in \N}\|\nabla \tilde u_{\e_k}\|_{L^p(B_1)}=\sup_{k\in \N}\|\nabla u_{\e_k}\|_{L^p(\Omega_{\e_k})} <\infty.
\ee
Up to a substraction of subsequence,
\be\label{wl-tu}
\tilde u_{\e_k} \to \tilde u \ \mbox{weakly in}\  W_0^{1,p}(B_1), \quad \mbox{as} \ k\to \infty.
\ee

\medskip

We firstly claim:
\begin{proposition}\label{prop:tu0}
The weak limit $\tilde u\in C^{0,1-\frac{3}{p}}(B_1)$ and $\tilde u(0)=0$.
\end{proposition}
\begin{proof}[Proof of Proposition \ref{prop:tu0}] By \eqref{wl-tu} and the fact $p>3$, Sobolev embedding and compact Sobolev embedding implies up to a substraction of subsequence that
\be\label{st-tu}
\tilde u\in C^{0,1-\frac{3}{p}}(B_1),\quad \tilde u_{\e_k} \to \tilde u \ \mbox{strongly in}\  C^{0,\l}( B_1), \quad \mbox{as} \ k\to \infty,
\ee
for any $\l<1-3/p$.

Since $\tilde u_{\e_k}=0$ on $\e_k T\ni 0$, the strong convergence in \eqref{st-tu} implies that $\tilde u(0)=0$.

\end{proof}

We secondly claim:
\begin{proposition}\label{prop:tu} The weak limit $\tilde u$ in \eqref{wl-tu} solves the Dirichlet problem of the Laplace equation in the unit ball:
\ba\label{eq-tu}
-\Delta \tilde u&=\dive f,\quad &&\mbox{in}~B_1,\\
\tilde u&=0,\quad &&\mbox{on} ~\d B_1.
\ea

\end{proposition}

\begin{proof}[Proof of Proposition \ref{prop:tu}] To show \eqref{eq-tu}, it is sufficient to prove
\be\label{wk-fm-tu}
\int_{B_1}\nabla\tilde u \cdot \nabla \vp \,dx =-\int_{B_1}f \cdot \nabla \vp \,dx \quad \mbox{for any $\vp\in C_c^\infty(B_1)$}.
\ee
Since $u_{\e_k}$ is a solution to \eqref{1}, the zero extension $\tilde u_{\e_k}$ satisfies
\be\label{wk-fm-tue}
\int_{B_1}\nabla\tilde u_{\e_k} \cdot \nabla \phi \,dx =-\int_{B_1}f \cdot \nabla \phi \,dx \quad \mbox{for any $\phi \in C_c^\infty(B_1\setminus {\e_k} T)$}.
\ee
Letting $k\to \infty $ in \eqref{wk-fm-tue} gives
\be\label{wk-fm-tu1}
\int_{B_1}\nabla\tilde u \cdot \nabla \psi \,dx =-\int_{B_1}f \cdot \nabla \psi \,dx \quad \mbox{for any $\psi \in C_c^\infty(B_1\setminus \{0\})$}.
\ee
We introduce a sequence of cut-off functions $\phi_n\in C^\infty(\R^3),\ n\in \Z_+$ satisfying
\be\label{cut-off2}
0\leq \phi_n\leq 1,\quad \phi_n =0 \ \mbox{in}\ B_{1/n}, \quad \phi_n =1 \ \mbox{on}\ \{x:|x|\geq 2/n\},\quad |\nabla\phi_n|\leq 2n.
\ee
Then for any $1\leq q\leq \infty$, we have the estimates
\be\label{cut-off20}
\|(1-\phi_n)\|_{L^q(\R^3)} \leq C \, n^{-\frac{3}{q}},\quad \|\nabla \phi_n\|_{L^q(\R^3)} \leq C\, n^{1-\frac{3}{q}}
\ee

\medskip

For any $\vp\in C_c^\infty(B_1)$, there holds
\ba\label{wk-fm-tu2}
&\int_{B_1}(\nabla\tilde u+f) \cdot \nabla \vp \,dx=\int_{B_1}(\nabla\tilde u +f)\cdot \nabla (\vp\phi_n) \,dx \\
&\quad - \int_{B_1} (\nabla\tilde u +f)\cdot \vp\nabla \phi_n \,dx+\int_{B_1} (\nabla\tilde u +f)\cdot (1-\phi_n)\nabla \vp \,dx\\
&=- \int_{B_1} (\nabla\tilde u +f)\cdot \vp\nabla \phi_n \,dx+\int_{B_1} (\nabla\tilde u +f)\cdot (1-\phi_n)\nabla \vp \,dx,
\ea
for which we used \eqref{wk-fm-tu1} in the second equality.

By \eqref{cut-off2} and \eqref{cut-off20}, we have
\ba\label{wk-fm-tu3}
&\left|\int_{B_1} (\nabla\tilde u +f)\cdot \vp\nabla \phi_n \,dx\right|\leq \|\nabla\tilde u +f\|_{L^p}\,\|\nabla \phi_n\|_{L^{p'}}\,\|\vp\|_{L^\infty} \leq C \, n^{1-\frac{3}{p'}},\\
&\left|\int_{B_1} (\nabla\tilde u +f)\cdot (1-\phi_n)\nabla \vp \,dx\right|\leq \|\nabla\tilde u +f\|_{L^p}\,\|(1- \phi_n)\|_{L^{p'}}\,\|\nabla \vp\|_{L^\infty} \leq C \,n^{-\frac{3}{p'}}.
\ea
The Lebegue norms in \eqref{wk-fm-tu3} are taken in $B_1$.  The choice $p>3$ implies $p'<3/2$ and furthermore $1-3/p'<-1$. This implies the quantities in \eqref{wk-fm-tu3} go to zero as $n\to \infty$. Thus passing $n\to \infty$ in \eqref{wk-fm-tu2} implies our desired result \eqref{wk-fm-tu}. We complete the proof of Proposition \ref{prop:tu}.

\end{proof}

Now we are ready to derive a contradiction. We recall the Green's function of the Laplace equation in the unit ball:
$$
G(x,y)=\Phi(x-y)-\Phi\left(|x|\left(\frac{x}{|x|^2}-y\right)\right),
$$
where $\Phi(x)=\a/|x|$ is the fundamental solution of the Laplace operator in $\R^3$. Then by Proposition \ref{prop:tu}, we have the expression
\be\label{ex-tu}
\tilde u(x)= \int_{B_1} G(x,y)\,\dive f(y)\,dy=\a\int_{B_1} \left(\frac{1}{|x-y|}-\frac{1}{\left|\frac{x}{|x|}-|x|y\right|}\right)\dive f(y)\,dy.\nn
\ee
This gives
\be\label{ex-tu0}
\tilde u(0)=\a\int_{B_1} \left(\frac{1}{|y|}-1\right)\dive f(y)\,dy,\nn
\ee
which is well defined due to our assumption that $\dive f \in L^q(B_1)$ for some $q>3/2$.

Applying Proposition \ref{prop:tu0} implies
\be\label{ex-tu1}
\int_{B_1} \left(\frac{1}{|y|}-1\right)\dive f(y)\,dy=0,\nn
\ee
which contradicts to \eqref{f-ex}. This means the assumption \eqref{est-con} is not true. We thus obtain \eqref{est12} and complete the proof of Theorem \ref{thm11}.


\section{Proof of Theorem \ref{thm1}}

To prove the first part of Theorem \ref{thm1}, it is sufficient to take $f(x)=(x_1,0,0)$ and to apply Theorem \ref{thm11}. Indeed, such $f(x)$ satisfies the assumptions in Theorem \ref{thm1} and Theorem \ref{thm11}; in particular,
\be\label{ex-tu10}
\int_{B_1} \left(\frac{1}{|y|}-1\right)\dive f(y)\,dy=\int_{B_1} \left(\frac{1}{|y|}-1\right)\,dy\neq 0.\nn
\ee

\medskip

Now we prove the second part of Theorem \ref{thm1} by duality arguments. Let $1< p<3/2$ and $f(x)=(x_1,0,0)\in C^\infty(\overline B_1;\R^3)$ fulfills the assumptions in Theorem \ref{thm11}. Since $T$ has $C^1$ boundary, then for any $0<\e<1$ there exists a unique solution $v_\e\in W_0^{1,p'}(\O_\e)$ to Dirichlet problem \eqref{1}. Since $3<p'<\infty$, by Theorem \ref{thm11}, we have
\be\label{est-ve}
\liminf_{\e\to 0}\|\nabla v_\e\|_{L^{p'}(\Omega_\e)}=\infty.
\ee
We will show that $f_\e$ defined below fulfills our request:
\be\label{def-fe}
f_\e:=\frac{|\nabla v_\e|^{p'-2}\nabla v_\e}{\|\nabla v_\e\|_{L^{p'}(\O_\e)}^{\frac{p'}{p}}}.\nn
\ee
Direct calculations gives
$$
\|f_\e\|_{L^p(\O_\e)}=1.
$$
Since the domain $T$ is $C^1$,  for any $0<\e<1$, there exists a unique solution $u_\e\in W_0^{1,p}(\O_\e)$ to Dirichlet problem \eqref{1} with source function $f_\e$. We have
\ba\label{est-ve1}
&\|\nabla u_\e\|_{L^{p}(\Omega_\e)}=\sup_{\|\phi\|_{L^{p'}(\Omega_\e)=1}}|\langle \nabla u_\e,\phi\rangle|\geq \|f\|_{L^{p'}(\Omega_\e)}^{-1}|\langle \nabla u_\e,f\rangle|\\
&=\|f\|_{L^{p'}(\Omega_\e)}^{-1}|\langle  \nabla u_\e, \nabla v_\e \rangle|=\|f\|_{L^{p'}(\Omega_\e)}^{-1}|\langle  f_\e, \nabla v_\e \rangle|=\|f\|_{L^{p'}(\Omega_\e)}^{-1} \|\nabla v_\e\|_{L^{p'}(\O_\e)}.
\ea
In \eqref{est-ve1} we used the fact that $v_\e$ and $u_\e$ satisfy Dirichlet problem \eqref{1} with right-hand side $\dive f$ and $\dive f_\e$ respectively. The estimate \eqref{est-ve} implies
\be\label{est-ue}
\liminf_{\e\to 0}\|\nabla u_\e\|_{L^{p}(\Omega_\e)}=\infty.\nn
\ee
This is exactly \eqref{est12}. We complete the proof of Theorem \ref{thm1}.

\section{Conclusions and perspectives}

In this paper, we gave a quite complete study for the uniformness of the $W^{1,p}$ estimates for the Dirichlet problem of the Laplace equation in the domain $\Omega_\e:=B_1\setminus \e T\subset \R^d$. Under certain assumptions on the regularity of $T$ (Lipchitz in three dimensions and $C^1$ in higher dimensions), we showed that for $d'<p<d$, there hold uniform $W^{1,p}$ estimates as $\e\to 0$; for any $d<p<\infty$, no matter how smooth the hole $T$ is, there exist smooth source functions $f\in C^\infty(\overline B_1;\R^d)$ such that the $W^{1,p}$ norms of the corresponding  solutions go to infinity as $\e$ goes to zero; finally for $1<p<d'$, there exit source functions $f_\e$ satisfying $\|f_\e\|_{L^p(\Omega_\e)}=1$ for any $0<\e<1$ such that the $W^{1,p}$ norms of the corresponding solutions go to infinity as $\e$ goes to zero.

However, the results here do not cover the case $p=d$ or $p=d'$ due to some technical difficulties. Particularly, in the proof of Lemma \eqref{lem:lap-Be}, we need to assume $p<d$ such that $|x|^{-p}$ is integrable in $B_{1/\e}$  (see \eqref{est-m21} - \eqref{est-m23}), and also in the proof of Proposition \ref{prop:v} we need to assume $p>d'$ to make sure that the quantity in \eqref{est-Phi7} is finite. Hence, the conclusion for the case $p=d$ or $p=d'$ is unclear.


{

\end{document}
\begin{thebibliography}{000}


\bibitem{ALL-NS1} G. Allaire. \newblock Homogenization of the Navier-Stokes equations in open sets perforated with tiny holes. I. Abstract framework, a volume
distribution of holes. \newblock {\em Arch. Ration. Mech. Anal.}, {\bf 113}(3) (1990), 209-259.


\bibitem{ALL-NS2} G. Allaire. \newblock Homogenization of the Navier-Stokes equations in open sets perforated with tiny holes. II. Noncritical sizes of the holes for a volume distribution and a surface distribution of holes. \newblock {\em Arch. Ration. Mech. Anal.}, {\bf 113}(3) (1990), 261-298.


\bibitem{BS}R.M. Brown, Z. Shen. \newblock Estimates for the Stokes operator in Lipschitz domains.
\newblock {\em Indiana Univ. Math. J.}, {\bf 44}(4) (1995), 1183-1206.

\bibitem{CP}L.A. Caffarelli, I. Peral. \newblock On $W^{1,p}$ estimates for elliptic equations in divergence form.
\newblock {\em Comm. Pure Appl. Math.}, {\bf 51}(1) (1998), 1-21.

\bibitem{FL1}E. Feireisl, Y. Lu. \emph{Homogenization of stationary Navier-Stokes equations in domains with tiny holes.} {\em J. Math. Fluid Mech.}, {\bf 17} (2015), 381-392.

\bibitem{FNT-Hom} E. Feireisl, A. Novotn\'y and T. Takahashi. \newblock Homogenization and singular limits for the complete Navier-Stokes-Fourier system.
\newblock {\em J. Math. Pures Appl.},  {\bf 94}(1) (2010), 33-57.


\bibitem{JK} D. Jerison, C. Kenig.  \newblock The inhomogeneous Dirichlet problem in Lipschitz domains.
\newblock {\em J. Funct. Anal.},  {\bf 130}(1) (1995), 161-219.

\bibitem{JM} W. J\"ager, A. Mikeli\'c. \newblock Homogenization of the Laplace equation in a partially perforated domain.
\newblock {\em In memory of Serguei Kozlov, volume 50 of Advances in Mathematics for Applied Sciences,} (1999), 259-284.

\bibitem{KS}H. Kozono, H. Sohr. \newblock New a priori estimates for the Stokes equations in exterior domains.
 \newblock {\em Indiana Univ. Math. J.}, {\bf 40}(1) (1991), 1-27.

\bibitem{Mas-Hom} N. Masmoudi. \newblock Homogenization of the compressible Navier-Stokes equations in a porous medium.
\newblock {\em ESAIM Control Optim. Calc. Var.}, {\bf 8} (2002), 885-906.

\bibitem{Mik} A. Mikeli\'{c}. \newblock {Homogenization of nonstationary Navier-Stokes equations in a domain with a grained boundary}.
  \newblock {\em Ann. Mat. Pura Appl.}, {\bf 158} (1991), 167-179.



 \end{thebibliography}
